\documentclass{amsart}
\usepackage{amssymb, a4wide, mathdots, url,hyperref,graphicx}
\usepackage[usenames]{color}
\usepackage{enumerate}
\usepackage{wasysym}
\usepackage{dsfont}
\usepackage{tikz-cd}
\usepackage{comment}
\usepackage{MnSymbol}
\usepackage{eqnarray, amsmath}
\newcommand{\C}{\mathbb{C}}
\newcommand{\N}{\mathbb{N}}
\newcommand{\Z}{\mathbb{Z}}

\newcommand{\Q}{\mathbb{Q}}
\newcommand{\F}{\mathbb{F}}
\newcommand{\A}{\mathbb{A}}

\newcommand{\Gal}{\operatorname{Gal}}

\newcommand{\ad}{\operatorname{ad}}
\newcommand{\ab}{\operatorname{ab}}

\renewcommand{\ss}{\operatorname{ss}}
\newcommand{\der}{\operatorname{der}}

\newcommand{\std}{\operatorname{std}}

\newcommand{\SO}{\mathrm{SO}}

\newcommand{\GL}{\mathrm{GL}}

\newcommand{\SL}{\mathrm{SL}}

\newcommand{\GO}{\mathrm{GO}}

\def\bG{\mathbf{G}}
\def\bH{\mathbf{H}}

\def\bT{\mathbf{T}}

\newcommand{\hto}{\hookrightarrow}

\newtheorem{theorem}{Theorem}[section]
\newtheorem*{theorem*}{Theorem}
\newtheorem{lemma}[theorem]{Lemma}
\newtheorem{definition}[theorem]{Definition}
\newtheorem{proposition}[theorem]{Proposition}
\newtheorem{corollary}[theorem]{Corollary}

\theoremstyle{remark}
\newtheorem{remark}[theorem]{Remark}

\title{Monodromy and irreducibility of type $A_1$ automorphic Galois representations}

\author{Chun-Yin Hui}
\address{Department of
 Mathematics, The University of Hong
 Kong, Pokfulam, Hong Kong}
\email{pslnfq@gmail.com}

\author{Wonwoong Lee}
\address{ Department of
 Mathematics, The University of Hong
 Kong, Pokfulam, Hong Kong}
\email{leeww041@hku.hk}


\thanks{}

\date{\today}

\begin{document}

\begin{abstract}
Let $K$ be a totally real field and $\pi$ be a regular algebraic polarized cuspidal automorphic representation of $\GL_n(\A_K)$.
Let $\{\rho_{\pi,\lambda}:\Gal_K\to\GL_n(\overline E_\lambda)\}_\lambda$ be
the compatible system of Galois representations attached to $\pi$ and denote by $\bG_\lambda$ the algebraic monodromy group of $\rho_{\pi,\lambda}$.
Suppose there exists $\lambda_0$ such that 
(a) $\rho_{\pi,\lambda_0}$ is irreducible; 
(b) $\bG_{\lambda_0}$ is connected and of type $A_1$;
and (c) the tautological representation of $\bG_{\lambda_0}$ 
is of a certain type. We prove that 
\begin{itemize}
\item $\bG_{\lambda,\C}\subset\GL_{n, \C}$ is independent of $\lambda$;
\item $\rho_{\pi,\lambda}$ is irreducible for all $\lambda$, and residually irreducible for almost all $\lambda$.
\end{itemize}
Moreover, if $K=\Q$ or $n$ is odd, we prove that the same conclusions hold without the assumption that $\pi$ is polarized. We also prove that if $K=\mathbb Q$, then the compatible system $\{\rho_{\pi,\lambda}\}_\lambda$ is constructed from certain two-dimensional modular compatible systems up to twist.
\end{abstract}

\maketitle

\section{Introduction}
Let $K$ be a totally real or CM field, and let $\pi$ be a regular algebraic cuspidal automorphic representation of $\GL_n(\A_K)$.
There is a \emph{compatible system} (Definition \ref{csdef}) attached to $\pi$
\begin{equation}\label{acs}
\{\rho_{\pi,\lambda}:\Gal_K\to\GL_n(\overline E_\lambda)\cong\GL_n(\overline\Q_\ell)\}_\lambda
\end{equation}
consisting of semisimple $\ell$-adic representations of $K$ defined over some number field $E$ that satisfies local-global compatibility away from $\ell$ \cite{HLTT16},\cite{Sch15},\cite{Var18}.
If $\pi$ is in addition polarized, \eqref{acs} is a \emph{strictly compatible system}
(Definition \ref{def_weakly compatible system}) and 
satisfies full local-global compatibility, see \cite[$\mathsection2.1$]{BLGGT14} and 
the references therein for details. It is a well-known conjecture in folklore that such automorphic Galois representations $\rho_{\pi,\lambda}$ are irreducible for all $\lambda$ (see \cite{Ram08}).
Although this conjecture has been extensively studied for small $n$ 
(see \cite[$\mathsection1$]{Hui23b} for details), for general $n$ and polarized $\pi$ it is proven that irreducibility holds for infinitely many $\lambda$ by Patrikis-Taylor \cite{PT15}, in which strict compatibility 
of \eqref{acs} and the potential automorphy theorems \cite{BLGGT14} serve as two key ingredients.
By combining the two ingredients with some $\lambda$-independence and big image results on monodromy groups in a (strictly) compatible system, 
the first author proved that $\rho_{\pi,\lambda}$ is irreducible for almost all $\lambda$
in the case that $n\leq 6$ and $\pi$ is polarized \cite{Hui23a}.
This strategy is then employed in \cite{Hui23b} to study $\lambda$-independence 
of algebraic monodromy groups of some $4$-dimensional irreducible
strictly compatible system of $\Q$. 

Let $\bG_\lambda \subset \GL_{n,\overline E_\lambda}$ denote the algebraic monodromy group of $\rho_{\pi,\lambda}$, defined as the Zariski closure of the image of $\rho_{\pi,\lambda}$ in $\GL_{n,\overline E_\lambda}$.
In this article, we apply the techniques from \cite{Hui23a,Hui23b} to prove
new case of the irreducibility conjecture for totally real $K$ and general $n$ (instead of small $n$) 
on the whole family $\{\rho_{\pi,\lambda}\}_\lambda$
if for some prime $\lambda_0$ the algebraic monodromy group $\bG_{\lambda_0}$
satisfies certain conditions.

We fix some notation and terminology.
Let $F$ be an algebraically closed field of characteristic zero and $G$ be a reductive group defined over $F$.
We say that $G$ is of \emph{type $A$} (resp. \emph{type $A_1$}) if every simple factor of the 
Lie algebra $\mathfrak{g}=\mathrm{Lie}(G)$ is isomorphic to some $\mathfrak{sl}_k$ (resp. $\mathfrak{sl}_2$).
For example, $\GL_2\times\GL_3$ is of type $A$ and $\GL_2\times\GL_2$ is of type $A_1$.
Denote by $G^{\der}$ the derived group $[G^\circ,G^\circ]$ of the identity component $G^\circ$.
If $G\subset\GL_{n,F}$ is a connected reductive subgroup
that is irreducible on the ambient space, 
the tautological representation $\rho$ of $G^{\der}$ is also irreducible
and admits an
exterior tensor decomposition 
\begin{equation}\label{etd}
(G^{\der},\rho)=(G_1G_2\cdots G_r,\rho_1 \otimes \rho_2 \otimes \cdots \otimes \rho_r)
\end{equation}
such that $G_i$ are almost simple factors of $G^{\der}$ and $(G_i,\rho_i)$ 
are irreducible representations for all $1\leq i\leq r$. We call $(G_i,\rho_i)$ for $1\leq i\leq r$ the \emph{basic factors} of 
the decomposition. For example, $(\SO_4,\std)=(\SL_2\SL_2,\std\otimes\std)$. If $F'$ is a field containing $F$, denote by $G_{F'}$ the base change $G\times_F F'$.

\begin{theorem}\label{thm_main theorem}
    Suppose $K$ is a totally real field and $\{\rho_{\pi, \lambda}: \mathrm{Gal}_{K} \to \mathrm{GL}_n(\overline E_{\lambda})\}_{\lambda}$ is the strictly compatible system of $K$ (defined over $E$) attached to a regular algebraic polarized cuspidal automorphic representation $\pi$ of $\mathrm{GL}_n(\mathbb A_{K})$ such that there exists a prime $\lambda_0$ satisfying the following conditions.
    \begin{enumerate}[\normalfont(a)]
        \item The representation $\rho_{\pi, \lambda_0}$ is irreducible. 
        \item The algebraic monodromy group $\bG_{\lambda_0}$ is connected and of type $A_1$.
        \item \begin{enumerate}[(1)]
				\item At most one basic factor in the exterior tensor decomposition of 
				the tautological representation of $\bG_{\lambda_0}^{\mathrm{der}}$ is $(\mathrm{SL}_{2},\mathrm{std})$,
				or 
				\item the tautological representation of $\bG_{\lambda_0}^{\mathrm{der}}$ is $(\SO_4,\std)$. 
				\end{enumerate}
    \end{enumerate}
    
    Then the following statements hold.
    \begin{enumerate}[\normalfont(i)]
        \item Embed $\overline E_\lambda$ in $\C$ for all $\lambda$. The conjugacy class of $\bG_{\lambda,\C}$ in $\mathrm{GL}_{n,\C}$ is independent of $\lambda$. 
        \item The representation $\rho_{\pi, \lambda}$ is irreducible for all $\lambda$.
				\item The representation $\rho_{\pi, \lambda}$ is residually irreducible for almost all $\lambda$.
   \end{enumerate}
\end{theorem}

Here, ``$\rho_{\pi,\lambda}$ is residually irreducible'' in (iii) means that the semisimplification of the mod $\ell$ reduction $$\overline{\rho}_{\pi,\lambda}^{\ss}: \mathrm{Gal}_K \to \mathrm{GL}_n(\overline\F_{\lambda})$$ 
is irreducible, where $\F_{\lambda}$ is the residue field of $E_{\lambda}$.  
We assume Theorem \ref{thm_main theorem}(c1) because we need to apply a result of Larsen-Pink \cite{LP90}
on when an irreducible faithful representation of a connected semisimple group
is determined by its formal character.
Since the standard representation $(\mathrm{SL}_{2},\mathrm{std})$ is the only possible 
$2$-dimensional basic factor of $\bG_{\lambda_0}^{\mathrm{der}}$, we have the following consequence.

\begin{corollary}
    Suppose $K$ is a totally real field and $\{\rho_{\pi, \lambda}: \mathrm{Gal}_{K} \to \mathrm{GL}_n(\overline E_{\lambda})\}_{\lambda}$ is the compatible system of $K$ (defined over $E$) attached to a regular algebraic polarized cuspidal automorphic representation $\pi$ of $\mathrm{GL}_n(\mathbb A_{K})$ such that there exists a prime $\lambda_0$ that satisfies the conditions Theorem~\ref{thm_main theorem}{\normalfont(a)},{\normalfont(b)}. If $n$ is odd or $n=2n'$ with an odd integer $n'$ or $n=4$,  then the conclusions of Theorem~\ref{thm_main theorem} hold.
\end{corollary}



Under certain conditions, we can drop the ``polarizability'' condition on $\pi$ in Theorem~\ref{thm_main theorem}.

\begin{theorem}\label{thm_main thm wo polarizability}
    Suppose $K$ is a totally real field and $\{\rho_{\pi, \lambda}: \mathrm{Gal}_{K} \to \mathrm{GL}_n(\overline E_{\lambda})\}_{\lambda}$ is the compatible system of $K$ (defined over $E$) attached to a regular algebraic cuspidal automorphic representation $\pi$ of $\mathrm{GL}_n(\mathbb A_{K})$ such that there exists a prime $\lambda_0$ that satisfies the conditions Theorem~\ref{thm_main theorem}{\normalfont(a)}--{\normalfont(c)}. We further assume either 
    \begin{enumerate}[\normalfont(a)]
        \item $K=\mathbb Q$, or
        \item $n$ is odd,
    \end{enumerate} 
    then the conclusions of Theorem~\ref{thm_main theorem} hold.
\end{theorem}

Indeed, we will prove that $\pi$ must be polarized under the assumptions in Theorem~\ref{thm_main thm wo polarizability}  (Proposition~\ref{prop_essential self-duality}). Then Theorem~\ref{thm_main thm wo polarizability} follows from Theorem~\ref{thm_main theorem}. 

When $K=\Q$, the compatible system $\{\rho_{\pi, \lambda}\}_\lambda$ can be constructed from two-dimensional modular compatible systems up to twist.

\begin{theorem}\label{thm_main theorem-classification}
    Suppose $\{\rho_{\pi, \lambda}: \mathrm{Gal}_{\mathbb Q} \to \mathrm{GL}_n(\overline E_{\lambda})\}_{\lambda}$ is the compatible system of $\Q$ (defined over $E$) attached to a regular algebraic cuspidal automorphic representation $\pi$ of $\mathrm{GL}_n(\mathbb A_{\mathbb Q})$ that satisfies the conditions Theorem~\ref{thm_main theorem}(a)--(c). After replacing the coefficient field $E$ with a larger field if necessary, 
		there exist strictly compatible systems of $2$-dimensional modular Galois representations 
		$\{f_{i, \lambda}: \mathrm{Gal}_{\mathbb Q} \to \mathrm{GL}_2(\overline{E}_{\lambda})\}_{\lambda}$ for $1\leq i\leq r$, 
	 a compatible system of $1$-dimensional representations $\{\chi_{\lambda}: \mathrm{Gal}_{\mathbb Q} \to \mathrm{GL}_1(\overline{E}_{\lambda})\}_{\lambda}$, and	$k_1,\ldots,k_r \in\N$ such that
    \begin{align*}
        \rho_{\pi, \lambda}=\mathrm{Sym}^{k_1}(f_{1, \lambda})\otimes\cdots \otimes \mathrm{Sym}^{k_r}(f_{r, \lambda}) \otimes \chi_{\lambda}
    \end{align*}
		 for all $\lambda$.
\end{theorem}

The integer $r$ above is the semisimple rank of $\bG_\lambda$, which is independent of $\lambda$ (Theorem~\ref{thm_formal bi-char indep}).

\begin{remark}\label{SP}
(Suggested by Stefan Patrikis)
The analogue of Theorem \ref{thm_main theorem-classification} for totally real fields is false!
There exist mixed parity non-CM Hilbert modular forms $\pi_1$ and $\pi_2$ (over $K$) such that 
the automorphic tensor product
$\pi=\pi_1 \boxtimes \pi_2$ is regular algebraic, polarized, cuspidal
with the algebraic monodromy group of $\rho_{\pi,\lambda}$ 
equal to $\GO_4\subset\GL_4$ for all $\lambda$ 
(i.e., conditions (a)--(c) of Theorem~\ref{thm_main theorem} are satisfied) \cite[$\mathsection2.6.2$]{Pat19}.
According to \cite[Proposition 2.6.7(ii)]{Pat19}, $\rho_{\pi,\lambda}$ is isomorphic 
to the tensor product $\rho_{1,\lambda}\otimes\rho_{2,\lambda}$, where $\rho_{1,\lambda}$ and $\rho_{2,\lambda}$ are two-dimensional, continuous, almost everywhere unramified Galois representations of $K$, no twist of either $\rho_{1,\lambda}$ or $\rho_{2,\lambda}$ 
is Hodge--Tate, and $\rho_{\pi,\lambda}$ is not a tensor product of geometric Galois representations.
Therefore, the four-dimensional compatible system $\{\rho_{\pi, \lambda}\}_\lambda$ cannot be the tensor product of two Hilbert modular compatible systems
up to a twist.
\end{remark}

We describe the remaining sections of this article.
Section~\ref{sec_prelim} introduces the fundamental objects studied or used in this article, including general (weakly/strictly) compatible systems, properties and conjectures of automorphic compatible systems, potential automorphy results, algebraic monodromy groups of a compatible system and their $\lambda$-independence. 
In Section~\ref{sec_ess self duality}, we prove that Theorem \ref{thm_main thm wo polarizability} follows from Theorem \ref{thm_main theorem} by exploiting the type $A_1$ condition and a result from \cite{BH24} on local algebraicity.
Theorems \ref{thm_main thm wo polarizability} and \ref{thm_main theorem-classification} are proven in Sections $4$ and $5$, respectively,
by combining various results of the representation theory of reductive groups
and Galois representations focusing on, for instance, $\lambda$-independence, big images, irreducibility, (potential) automorphy, liftings.

\section{Preliminaries}\label{sec_prelim}

\subsection{Compatible systems}

Let $K$ and $E$ be number fields. We denote by $\mathrm{Gal}_{K}$ the absolute Galois group of $K$, and also denote by $\Sigma_K$ and $\Sigma_E$ the set of finite places of $K$ and $E$, respectively. We use the notation $p$ and $\ell$ for rational primes, and we let $v \in \Sigma_K$ be a prime above $p$ and $\lambda \in \Sigma_E$ be a prime above $\ell$.

An $\ell$-adic Galois representation 
$\rho: \mathrm{Gal}_K \to \mathrm{GL}_n(\overline{E}_{\lambda})\cong\GL_n(\overline\Q_\ell)$
of $K$ is a continuous group homomorphism. It is said to be $E$-rational if there exist a finite set of places $S \subset \Sigma_K$ for which the following two conditions are satisfied.
\begin{enumerate}[\normalfont(a)]
    \item If $v \not\in S$, then $\rho$ is unramified at $v$.
    \item The characteristic polynomial $\Phi_{\rho, v}(T)$ of $\rho$ at $v$ is defined over $E$, i.e.,
    \begin{align}\label{Frobpoly}
        \Phi_{\rho, v}(T):=\det(\rho(\mathrm{Frob}_v)-T\cdot \mathrm{id}) \in E[T].
    \end{align}
\end{enumerate}

For a family of $n$-dimensional $\ell$-adic Galois representations 
$$\{\rho_{\lambda} : \mathrm{Gal}_{K} \to \mathrm{GL}_n(\overline{E}_{\lambda})\}_{\lambda \in \Sigma_E}$$
indexed by $\Sigma_E$,
we define a {\it compatible system} as follows.

\begin{definition}\label{csdef}
    The family $\{\rho_{\lambda} : \mathrm{Gal}_{K} \to \mathrm{GL}_n(\overline{E}_{\lambda})\}_{\lambda}$ is called a compatible system of $\ell$-adic Galois representations of $K$ defined over $E$ if there exist a finite set of places $S \subset \Sigma_K$ and  a polynomial $\Phi_v(T)\in E[T]$ for each $v\in\Sigma_K\backslash S$ such that
		the following conditions hold.
    \begin{enumerate}[\normalfont(a)]
        \item For all primes $v \not\in S \cup \{\lambda \mid \ell\}$, the representation $\rho_{\lambda}$ is unramified at $v$.
        \item For all primes $v \not\in S \cup \{\lambda \mid \ell\}$, the characteristic polynomial $\Phi_{\rho_\lambda,v}(T)$ in \eqref{Frobpoly} is equal to $\Phi_v(T)$.
    \end{enumerate}
\end{definition}

In particular, $\rho_\lambda$ is $E$-rational for all $\lambda\in\Sigma_E$.
We note that several fundamental operations such as direct sum, tensor product, dual, restriction to subgroup or induction from subgroup, coefficient extension, and restriction of scalars can be applied to compatible systems. These operations are described precisely in \cite[$\mathsection$2.4.4]{Hui23a}. For our purpose, we particularly give the definitions of coefficient extension and restriction of scalars (in case $\rho_\lambda(\Gal_K)\subset\GL_n(E_\lambda)$ for all $\lambda$) as follows.

\begin{definition}\label{def_coe ext and res  scalar}
    Let $\rho_\bullet:=\{\rho_{\lambda} : \mathrm{Gal}_{K} \to \mathrm{GL}_n(E_{\lambda})\}_{\lambda}$ be a compatible system of $K$ defined over $E$ and $E''/E/E'$ be number fields.
    \begin{enumerate}[\normalfont(i)]
        \item (Coefficient extension) $\rho_\bullet \otimes_E E'':=\{\rho_{\lambda} \otimes_{E_{\lambda}} E''_{\lambda''} :~ \lambda'' \mid \lambda\}_{\lambda'' \in \Sigma_{E''}}$. 
        \item (Restriction of scalars)  
        $$\mathrm{Res}_{E/E'}(\rho_\bullet):=\left\{\bigoplus_{\lambda \mid \lambda'} \rho_{\lambda}: \mathrm{Gal}_K \to \prod_{\lambda \mid \lambda'} \mathrm{GL}_n(E_{\lambda}) \subset \mathrm{GL}_{n[E:E']}(E'_{\lambda'})\right\}_{\lambda' \in \Sigma_{E'}}.$$
    \end{enumerate}
\end{definition}

\

We also introduce the concept of a weakly/strictly compatible system of $\ell$-adic Galois representations of $K$ over $E$.

\begin{definition}\label{def_weakly compatible system}
    A family $\{\rho_{\lambda} : \mathrm{Gal}_{K} \to \mathrm{GL}_n(\overline{E}_{\lambda})\}_{\lambda}$ is called a weakly compatible system of $\ell$-adic Galois representations of $K$ over $E$ if it is compatible and the following conditions are satisfied.
    \begin{enumerate}[\normalfont(a)]
        \item If $v$ is above $\ell$, then $\rho_{\lambda}|_{\mathrm{Gal}_{K_v}}$ is de Rham and is crystalline further if $v \not\in S$.
        \item For each embedding $\tau: K\hookrightarrow \overline E$, the 
				$\tau$-Hodge-Tate numbers  $\mathrm{HT}_{\tau}(\rho_{\lambda})=\mathrm{HT}_{\tau}(\rho_{\lambda}|_{\mathrm{Gal}_{K_v}})$ of $\rho_{\lambda}$ is independent of $\lambda$ and any $\overline E\hookrightarrow \overline E_\lambda$ over $E$.
    \end{enumerate}
    A weakly compatible system is called a strictly compatible system if an extra condition holds.
    \begin{enumerate}[\normalfont(a)]
        \setcounter{enumi}{2}
        \item If $v$ is not above $\ell$, then the semisimplified Weil-Deligne representation $\iota \mathrm{WD}(\rho_{\lambda}|_{\mathrm{Gal}_{K_v}})^{F-ss}$ is independent of $\lambda$
				and $\iota:\overline{E}_\lambda\stackrel{\cong}{\rightarrow} \C$.
    \end{enumerate}
\end{definition}

For the definition and basic properties of $\tau$-Hodge-Tate numbers, 
we refer the reader to \cite[$\mathsection$1]{PT15}. When $K=\mathbb Q$, 
these are the same as the usual Hodge-Tate numbers.
A (weakly/strictly) compatible system $\{\rho_\lambda\}_\lambda$
is said to be semisimple if $\rho_\lambda$ is semisimple for all $\lambda$.

\begin{definition}\label{regularalg}
    An $\ell$-adic representation $\rho_{\lambda} : \mathrm{Gal}_{K} \to \mathrm{GL}_n(\overline{E}_{\lambda})$ is said to be regular (or regular algebraic) if it is unramified at all but finitely many primes, and for any $v$ above $\ell$, $\rho_{\lambda}|_{\mathrm{Gal}_{K_v}}$ is de Rham and $\mathrm{HT}_{\tau}(\rho_{\lambda}|_{\mathrm{Gal}_{K_v}})$ consists of $n$ distinct numbers for any $\tau: K \hookrightarrow \overline E_{\lambda}$.
\end{definition}

\subsection{Automorphic Galois representations}\label{sec_automorphic Gal repn}

Let $K$ be a totally real or CM field and let $K^+$ be the maximal totally real subfield of $K$.
We refer the reader to \cite[$\mathsection$2.1]{BLGGT14} and \cite[$\mathsection$2]{CG13} for the definition of a \emph{regular algebraic cuspidal automorphic representation} $\pi$ of $\mathrm{GL}_n(\mathbb A_{K})$. 
We say that $\pi$ is \emph{polarized} if there is 
a continuous character $\xi: \mathbb A_{K^+}^{\times}/(K^+)^{\times} \to \mathbb C^{\times}$ 
such that 
\begin{itemize}
    \item $\xi_v(-1)$ is independent of $v \mid \infty$, and 
    \item $\pi^c \simeq \pi^{\vee} \otimes (\xi \circ \mathbf N_{K/K^+}\circ \det)$,
\end{itemize}
where $\pi^c$ denotes the composition of $\pi$ with complex conjugation on $\mathrm{GL}_n(\mathbb A_K)$. In the case that $K$ is imaginary, we further suppose that $\xi_v(-1)=(-1)^n$ for all $v \mid \infty$.

For a regular algebraic cuspidal automorphic representation $\pi$ of $\mathrm{GL}_n(\mathbb A_K)$, we can attach a semisimple compatible system 
$\{\rho_{\pi,\lambda}: \mathrm{Gal}_{K} \to \mathrm{GL}_n(\overline{E}_{\lambda})\cong\GL_n(\overline\Q_\ell)\}_{\lambda}$ due to many celebrated works. We refer to the paragraphs before and after \cite[Theorem 2.1.1]{BLGGT14} for the polarized case and \cite{HLTT16},\cite{Sch15} in general.
These $\ell$-adic representations $\rho_{\pi,\lambda}$ are conjecturally irreducible.

Suppose that $\pi$ is polarized. Then the system $\{\rho_{\pi,\lambda}\}_\lambda$ is strictly compatible and 
we may assume that it is given by
$$
\{\rho_{\pi, \lambda} : \mathrm{Gal}_{K} \to \mathrm{GL}_n(E_{\lambda})\}_{\lambda}
$$
after replacing $E$ with a sufficiently large CM field
(see \cite[Lemmas 1.2 and 1.4]{PT15} and \cite[Lemma 5.3.1(3)]{BLGGT14}). 
Moreover, the $\tau$-Hodge-Tate numbers of $\rho_{\lambda}$ at $v$ above $\ell$ are given by
\begin{align*}
    \mathrm{HT}_{\tau}(\rho_{\pi, \lambda})=\{a_{\tau,1}+(n-1), a_{\tau,2}+(n-2), \ldots, a_{\tau,n}\},
\end{align*}
where $a=(a_{\tau,i}) \in (\mathbb Z^n)^{\mathrm{Hom}(K,\mathbb C),+}$ denotes the weight of $\pi$. In particular, 
the system $\{\rho_{\pi, \lambda}\}_{\lambda}$ is regular. 
If $\pi$ is not polarized, it is not even known whether the local representations of
$\rho_{\pi,\lambda}$ at $v$ above $\ell$ are Hodge-Tate.

Conversely, it is asserted by the Fontaine-Mazur-Langlands conjecture that any irreducible $\ell$-adic Galois representation of $\mathbb Q$ which is unramified outside  a finite set $S$ of places and is de Rham locally at $\ell$ comes from a cuspidal automorphic representation of $\mathrm{GL}_n(\mathbb A_{\mathbb Q})$ \cite{FM95, Lan79} (see also \cite{Clo90, Tay04}). This property, where Galois representations correspond to automorphic representations, is called the {\it automorphy}. The following is one of the theorems on potential automorphy of regular algebraic Galois representations over a totally real field with some additional conditions.

\begin{theorem}{\cite[Theorem C]{BLGGT14}}\label{thm_BLGGT14 Thm C}
    Suppose that $K$ is a totally real field. Let $n$ be a positive integer and $\ell \geq 2(n+1)$ be a rational prime. Let
    \begin{align*}
        \rho_{\ell}:\mathrm{Gal}_K \to \mathrm{GL}_n(\overline{\mathbb Q}_{\ell})
    \end{align*}
    be a continuous representation. We denote by $\overline{\rho}^{\ss}_\ell$ the semisimplification of the reduction of $\rho_\ell$. Suppose the following conditions are satisfied.
    \begin{enumerate}[\normalfont(1)]
        \item (Unramified almost everywhere) $\rho_{\ell}$ is unramified at all but finitely many primes.
        \item (Odd essential self-duality) Either $\rho_{\ell}$ maps to $\mathrm{GSp}_n$ with totally odd multiplier or it maps to $\mathrm{GO}_n$ with totally even multiplier.
        \item (Potential diagonalizability and regularity) $\rho_{\ell}$ is potentially diagonalizable (and hence potentially crystalline) at each prime $v$ of $K$ above $\ell$ and for each embedding $\tau : K \to \overline{\mathbb Q}_{\ell}$ it has $n$ distinct  $\tau$-Hodge-Tate numbers.
        \item (Irreducibility) $\overline{\rho}_{\ell}^{\mathrm{ss}}|_{\mathrm{Gal}_{K(\zeta_{\ell})}}$ is irreducible, where $\zeta_{\ell}:=e^{2\pi i/\ell}$ the primitive $\ell$th root of unity.
    \end{enumerate}
    Then we can find a finite Galois totally real extension $K'/K$ such that $\rho_{\ell}|_{\mathrm{Gal}_{K'}}$ is attached to a regular algebraic polarized cuspidal automorphic representation of $\mathrm{GL}_n(\mathbb A_{K'})$. Moreover, $\rho_{\ell}$ is part of a strictly compatible system of $K$.
\end{theorem}

The following is a useful proposition for our purpose which is proven in \cite{Hui23b} by combining Theorem~\ref{thm_BLGGT14 Thm C} and the results in \cite{Hui23a}.

\begin{proposition}{\cite[Proposition 2.12(b)]{Hui23b}}\label{prop_SO3 potential automorphy}
    Let $\{\rho_{\lambda}:\mathrm{Gal}_{K} \to \mathrm{GL}_n(E_{\lambda})\}_{\lambda}$ be a strictly compatible system of a totally real field $K$ defined over $E$. For a sufficiently large $\ell$, if $\sigma_{\lambda}: \mathrm{Gal}_{K} \to \mathrm{GL}(W_{\lambda})$ is a regular $3$-dimensional $\ell$-adic subrepresentation of $\rho_{\lambda} \otimes \overline{E}_{\lambda}$ and the derived subgroup $G_{\sigma_{\lambda}}^{\mathrm{der}}$ of its algebraic monodromy group  is $\mathrm{SO}_3$ (as a group embedded in $\mathrm{GL_3})$, then there is a finite Galois totally real extension $K'/K$ such that $\sigma_{\lambda}|_{\mathrm{Gal}_{K'}}$ is attached to a regular algebraic polarized cuspidal automorphic representation of $\mathrm{GL}_3(\mathbb A_{K'})$. Moreover, for such $\lambda$, $\sigma_{\lambda}$ is a part of a strictly compatible system of Galois representations of $K$.
\end{proposition}

On the other hand, Patrikis--Taylor \cite{PT15} proved an irreducibility result for compatible systems attached to cuspidal automorphic representations as follows.

\begin{theorem}{\cite[Theorem 1.7]{PT15}}\label{thm_PT15}
    Suppose that $K$ is a totally real or CM field and $\pi$ is a regular algebraic polarized cuspidal automorphic representation of $\mathrm{GL}_n(\mathbb A_K)$. Let $\{\rho_{\pi,\lambda}:\mathrm{Gal}_K \to \mathrm{GL}_n(\overline E_{\lambda})\}_{\lambda}$ be the compatible system attached to $\pi$. Then there is a positive Dirichlet density set $\mathcal L$ of rational primes such that if a prime $\lambda \in \Sigma_{E}$ divides some $\ell \in \mathcal L$, then $\rho_{\pi,\lambda}$ is irreducible.
\end{theorem}

\subsection{Formal bi-characters}

Consider an algebraically closed field $F$ of characteristic zero and a reductive subgroup $G$ of the linear algebraic group $\mathrm{GL}_{n, F}$. Let $G^{\circ}$ be the identity component of $G$ and 
$G^{\der}:=[G^\circ, G^\circ]$. 
The formal character of $G$ is defined as the conjugacy class of subtori $[T_1]$ of $\mathrm{GL}_{n, F}$ such that some maximal torus of $G$ is in this conjugacy class.
Let $T$ be a maximal torus of $G$ and denote by $\mathbb{X}(T):=\mathrm{Hom}(T,\mathbb{G}_m)$ 
the character group of $T$. The formal character of $G$ corresponds 
to a multiset of elements of $\mathbb{X}(T)$.
The intersection $T':=G^{\der}\cap T$ is a maximal torus of the semisimple group $G^{\der}$.
The formal bi-character of $G$ is defined as the conjugacy class of 
a chain of subtori $[T_2\subset T_1]$ of $\mathrm{GL}_{n, F}$ such that $T'\subset T$ 
is in this conjugacy class. Both definitions are independent of the choice of the maximal torus of $G$.

Now, suppose that $G_1$ and $G_2$ are reductive subgroups of respectively $\mathrm{GL}_{n,F_1}$ and $\mathrm{GL}_{n,F_2}$, where $F_1$ and $F_2$ are fields of characteristic zero.
Let $F$ be an algebraically closed field containing $F_1$ and $F_2$.
We say that the formal characters (resp. bi-characters) of $G_1$ and $G_2$ are the same
if the formal characters (resp. bi-characters) of the base change $G_{1,F}$ and $G_{2,F}$ in $\GL_{n,F}$ coincide.
This definition is independent of the choice of $F$.


\subsection{$\lambda$-independence of algebraic monodromy groups}

Let $K$ and $E$ be number fields and $\{\rho_\lambda:\Gal_K\to\GL_n(\overline E_\lambda)\}_\lambda$ 
be a semisimple compatible system of $K$ defined over $E$. The algebraic monodromy group $\bG_\lambda$
of $\rho_\lambda$ is a reductive subgroup of $\GL_{n,\overline E_\lambda}$.
Serre \cite{Ser81} (see also \cite{LP92}) proved the following important theorem regarding the $\lambda$-independence of $\bG_\lambda$. 

\begin{theorem}\label{thm_Serre} (Serre)
    Let $\{\rho_{\lambda}\}_{\lambda}$ be a semisimple compatible system.
    \begin{enumerate}[\normalfont(i)]
        \item The component group $\bG_{\lambda}/\bG_{\lambda}^{\circ}$ is independent of $\lambda$.
        \item The formal character of $\bG_{\lambda} \subset \mathrm{GL}_{n,\overline E_\lambda}$ is independent of $\lambda$. In particular, the rank of $\bG_{\lambda}$ is independent of $\lambda$. 
    \end{enumerate}
\end{theorem}

Theorem~\ref{thm_Serre}(ii) was generalized to the theorem below by the first author of this paper.

\begin{theorem}{\cite[Theorem 3.19, Remark 3.22]{Hui13}}\label{thm_formal bi-char indep}
    Let $\{\rho_{\lambda}\}_{\lambda}$ be a semisimple compatible system. 
		Then the formal bi-character of $\bG_{\lambda} \subset \mathrm{GL}_{n,\overline E_\lambda}$ is independent of $\lambda$. In particular, the rank and the semisimple rank of $\bG_{\lambda}$ are both independent of $\lambda$.
\end{theorem}

Let $\lambda_1$ and $\lambda_2$ be two primes in $\Sigma_E$.
For $i=1,2$, take a maximal torus $\bT_{\lambda_i}$ of $\bG_{\lambda_i}$ 
and let $\bT_{\lambda_i}'$ be $\bG_{\lambda_i}^{\der}\cap\bT_{\lambda_i}$;
embed $\overline E_{\lambda_i}$ into $\C$.
By Theorem \ref{thm_formal bi-char indep}, there exists $g\in\GL_n(\C)$ such that 
$$g(\bT_{\lambda_1}'\subset\bT_{\lambda_1})g^{-1}=(\bT_{\lambda_2}'\subset\bT_{\lambda_2}),$$
which induces isomorphisms between the character groups $\mathbb{X}(\bT_{\lambda_1})\simeq \mathbb{X}(\bT_{\lambda_2})$
and $\mathbb{X}(\bT_{\lambda_1}')\simeq \mathbb{X}(\bT_{\lambda_2}')$.

\section{Polarizability of $\pi$}\label{sec_ess self duality}


The goal of this section is to prove that the automorphic representation $\pi$ in 
Theorem~\ref{thm_main thm wo polarizability} is polarized
(Proposition~\ref{prop_essential self-duality}). By using the type $A_1$-structure of the algebraic monodromy group, we will prove that the Galois representation $\rho_{\pi,\lambda_0}$ 
is self-dual up to a certain character $\chi'$. Furthermore, we will prove that $\chi'$ is locally algebraic by \cite{BH24}, so it corresponds to an algebraic Hecke character. This allows us to conclude that $\pi$ is polarized under the assumptions of Theorem ~\ref{thm_main thm wo polarizability}.

\subsection{Local algebraicity}

In this subsection, we briefly introduce the notion of locally algebraic abelian representation. For further details and proofs of the results in this subsection, we refer the reader to \cite[Chapters II and III]{Ser98}.

Let $K$ be a number field, and define the $[K:\mathbb Q]$-dimensional $\mathbb Q$-torus 
$T:=\mathrm{Res}_{K/\mathbb Q}(\mathbb G_{m/K})$ by restriction of scalars. 
If $A$ is an $\Q$-algebra, then the group of $A$-points of $T$ is $(K\otimes_\Q A)^\times$.
In particular,
\begin{align*}
    T(\Q_\ell)=(K \otimes \mathbb Q_{\ell})^{\times}=\prod_{v \mid \ell} K_v^{\times}.
\end{align*}
Using the Artin reciprocity map, we obtain a homomorphism
\begin{align*}
    i_{\ell}: (K \otimes \mathbb Q_{\ell})^{\times} \to \prod_{v \mid \ell} K_v^{\times} \to \mathbb A_K^{\times} \to \mathrm{Gal}_K^{\mathrm{ab}},
\end{align*}
where $\mathbb A_K^{\times}$ is the group of id\`{e}les of $K$.

\begin{definition}
    An abelian $\ell$-adic representation $\rho: \mathrm{Gal}_K^{\ab} \to \GL_n(\overline \Q_\ell)$ is locally algebraic if there is a $\overline\Q_\ell$-algebraic morphism
    \begin{align*}
        f_\ell: T_{\overline\Q_\ell} \to \GL_{n,\overline\Q_\ell}
    \end{align*}
    such that $\rho \circ i_{\ell}(x)=f_\ell(x^{-1})$ for 
		$x \in T(\mathbb Q_{\ell})\subset T(\overline\Q_\ell)$ that is sufficiently close to $1$ with respect to the $\ell$-adic Lie group structure on $T(\mathbb Q_{\ell})$.
\end{definition}

Similarly, we can define the local algebraicity of a local representation. Let $K_v$ be a finite field extension of $\mathbb Q_{\ell}$. Define similarly $T_v:=\mathrm{Res}_{K_v/\mathbb Q_{\ell}}(\mathbb G_{m,K_v})$ as a $\mathbb Q_{\ell}$-torus. We have
\begin{align*}
    T_v(\mathbb Q_{\ell})=K_v^{\times}
\end{align*}
and
\begin{align*}
    i_{v}:K_v^{\times} \to \mathrm{Gal}_{K_v}^{\mathrm{ab}}
\end{align*}
given by local class field theory.

\begin{definition}
    An abelian $\ell$-adic representation $\rho: \mathrm{Gal}_{K_v}^{\ab} \to \GL_n(\overline\Q_\ell)$ is locally algebraic if there is a $\overline{\mathbb Q}_{\ell}$-algebraic morphism
    \begin{align*}
        f_v: T_{v,\overline\Q_\ell} \to \GL_{n,\overline\Q_\ell}
    \end{align*}
    such that $\rho \circ i_{v}(x)=f_v(x^{-1})$ for $x \in T_v(\mathbb Q_{\ell})\subset T_v(\overline\Q_\ell)$ that is sufficiently close to $1$ with respect to the $\ell$-adic Lie group structure on $T_v(\mathbb Q_{\ell})$.
\end{definition}

The local algebraicity of an abelian $\ell$-adic representation of $K$ is determined solely by the local algebraicity of its local representations. 

\begin{proposition}
    An abelian $\ell$-adic representation $\rho: \mathrm{Gal}_K^{\ab} \to \mathrm{GL}_n(\overline\Q_\ell)$ of a number field $K$ is locally algebraic if and only if its local representation $\rho|_{\mathrm{Gal}_{K_v}^{\ab}}$ is locally algebraic for all $v$ above $\ell$. Moreover, it is equivalent to $\rho|_{\mathrm{Gal}_{K_v}^{\ab}}$ being Hodge-Tate for all 
		$v$ above $\ell$.
\end{proposition}

For potentially abelian (and in particular, abelian) representations, being Hodge-Tate is equivalent to being de Rham (\cite[$\mathsection$6]{FM95}). We also note that locally algebraic $\ell$-adic characters correspond to the algebraic Hecke characters.

The notion of $E$-rationality is related to local algebraicity as follows.

\begin{proposition}\label{prop_local algebraicity and compatible system}
    Let $\rho: \mathrm{Gal}_K^{\ab} \to \mathrm{GL}_n(\overline{\mathbb Q}_{\ell})$ be an abelian semisimple $\ell$-adic Galois representation. Then the following statements are equivalent.
    \begin{enumerate}[\normalfont(1)]
        \item $\rho$ is locally algebraic.
        \item $\rho$ is part of a semisimple compatible system of $K$ defined over some number field $E \subset \overline{\mathbb Q}_{\ell}$.
        \item $\rho$ is $E$-rational for some number field $E$.
    \end{enumerate}
\end{proposition}

\begin{proof}
    The proof is provided in \cite{Ser98} and is also formulated in \cite[$\mathsection$2.4]{BH24}.
\end{proof}

\subsection{Weak abelian direct summand and polarizability}

In a recent paper by B\"{o}ckle and the first author \cite{BH24}, the concept of weak abelian direct summands was introduced.

\begin{definition}
    Let $K$ be a number field, and let $\rho: \mathrm{Gal}_K \to \mathrm{GL}_n(\overline{\mathbb Q}_{\ell})$ and $\psi: \mathrm{Gal}_K \to \mathrm{GL}_m(\overline{\mathbb Q}_{\ell})$ be semisimple $\ell$-adic representations, unramified outside finite $S_{\rho}$ and $S_{\psi}$ in $\Sigma_K$, respectively. We say that $\psi$ is a weak abelian direct summand of $\rho$ if $\psi$ is abelian and the set
    \begin{align*}
        S_{\psi \mid \rho}:=\{v \in \Sigma_K : \Phi_{\psi, v}(T) \text{ divides } \Phi_{\rho, v}(T)\}
    \end{align*}
    is of Dirichlet density one, where $\Phi_{\rho, v}(T)$ is defined in \eqref{Frobpoly}.
\end{definition}

In particular, if $\psi$ is an abelian subrepresentation of $\rho$ then it is a weak abelian direct summand of $\rho$.
One of the main results in \cite{BH24} is the following theorem, which states that 
weak abelian direct summands of $E$-rational representations are locally algebraic.

\begin{theorem}{\cite[Theorem 1.2]{BH24}}\label{thm_weak abelian direct summand-locally algebraic}
    Let $K$ and $E \subset \overline{\mathbb Q}_{\ell}$ be number fields, and let $\rho_{\ell}: \mathrm{Gal}_K \to \mathrm{GL}_n(\overline{\mathbb Q}_{\ell})$ be an $E$-rational semisimple $\ell$-adic Galois representation of $K$. If $\psi_{\ell}$ is a weak abelian direct summand of $\rho_{\ell}$, then it is locally algebraic. 
\end{theorem}

We now use Theorem~\ref{thm_weak abelian direct summand-locally algebraic} to prove that 
the automorphic representation $\pi$ in Theorem~\ref{thm_main thm wo polarizability} is polarized (see $\mathsection2.2$ for definition).

\begin{proposition}\label{prop_essential self-duality}
    Let $K$ be a totally real field and $\{\rho_{\pi,\lambda}: \mathrm{Gal}_{K} \to \mathrm{GL}_n(\overline E_{\lambda})\}_{\lambda}$ be the compatible system attached to a regular algebraic cuspidal automorphic representation $\pi$ of $\mathrm{GL}_n(\mathbb A_K)$ that satisfies the assumptions in Theorem~\ref{thm_main thm wo polarizability}. Then $\pi$ is polarized.
\end{proposition}

\begin{proof}
  Since $\rho_{\pi,\lambda_0}$ is irreducible with connected algebraic monodromy group $\bG_{\lambda_0}$, 
		by Schur's lemma, we have two cases: (a) $\bG_{\lambda_0}=\bG_{\lambda_0}^{\mathrm{der}}$, or 
		(b) $\bG_{\lambda_0}=\mathbb G_m \bG_{\lambda_0}^{\mathrm{der}}$.	
		We claim that 
    \begin{align}\label{eqn_self duality of Gal repn}
        \rho_{\pi,\lambda_0} \simeq \rho_{\pi,\lambda_0}^{\vee} \otimes \epsilon^{1-n}\chi
    \end{align}
    for some one-dimensional $\ell_0$-adic representation $\chi$ of $K$, where $\epsilon$ is the $\ell_0$-adic cyclotomic character of $K$. Since every representation of $\bG_{\lambda_0}^{\mathrm{der}}$ (of type $A_1$) is self-dual,
this holds in case (a) by taking $\chi=\epsilon^{n-1}$. 	
		To establish \eqref{eqn_self duality of Gal repn} in case (b), it suffices to show that $\iota \simeq \iota^\vee \otimes \alpha$, where $\iota$ is the faithful irreducible representation of $\bG_{\lambda_0} = \mathbb{G}_m \bG_{\lambda_0}^{\mathrm{der}}$, and $\alpha$ is a character of $\bG_{\lambda_0}$. 	Without loss of generality, assume $\bG_{\lambda_0}^{\mathrm{der}}$
		is of rank $r\geq 1$. Consider the irreducible representation 
$$\Phi:=\mathrm{Sym}^{k_1} \otimes \cdots \otimes \mathrm{Sym}^{k_r}: \mathrm{GL}_2 \times \cdots \times \mathrm{GL}_2 \to \GL_n$$
that factors through $\bG_{\lambda_0}\subset\GL_n$. It follows that
    \begin{align*}
        \Phi=\Phi^{\vee}\otimes(\mathrm{det}^{k_1} \otimes \cdots \otimes \mathrm{det}^{k_r}).    
    \end{align*}
    It remains to show that the character $\mathrm{det}_1^{k_1} \otimes \cdots \otimes \mathrm{det}_r^{k_r}$ of  $\mathrm{GL}_2 \times \cdots \times \mathrm{GL}_2$ factors through $\bG_{\lambda_0}$. Since any element $(g_1,g_2,\ldots,g_r)\in\mathrm{ker}\Phi$ is given by an element $(c_1,c_2,...,c_r)\in\prod_{i=1}^r\mathbb{G}_m$ 
		satisfying $c_1^{k_1}c_2^{k_2}\cdots c_r^{k_r}=1$, we have
    \begin{align*}
        \mathrm{det}(g_1)^{k_1}\cdots\mathrm{det}(g_r)^{k_r}=c_1^{2k_1}c_2^{2k_2}\cdots c_r^{2k_r}=(c_1^{k_1}c_2^{k_2}\cdots c_r^{k_r})^2=1.
    \end{align*}
		We obtain \eqref{eqn_self duality of Gal repn} and write $\chi'=\epsilon^{1-n}\chi$ as the 
		Galois character corresponding to $\alpha$.
    
    Thus, proving that $\chi$ is locally algebraic implies that it corresponds to an algebraic Hecke character 
		$\xi_{\chi}:\mathbb A_K^{\times}/K^{\times} \to \mathbb C^{\times}$ 
		such that $\pi\simeq \pi^\vee\otimes (\xi_{\chi}\circ\det)$ holds by \eqref{eqn_self duality of Gal repn}
		and the strong multiplicity one theorem \cite{JS81}. 
        

		Let us prove that $\chi$ is locally algebraic. We denote the underlying vector space of $\rho_{\pi,\lambda_0}$ by $V$. From \eqref{eqn_self duality of Gal repn}, we have
    \begin{align*}
        V \otimes V=V\otimes V^{\vee} \otimes \chi'=\mathrm{End}(V) \otimes \chi'
    \end{align*}
    as $\ell_0$-adic representations (i.e., the equality is in the sense of $\mathrm{Gal}_K$-equivariant isomorphisms). Note that $\mathrm{End}(V)$ contains a $1$-dimensional trivial subrepresentation, namely the subrepresentation generated by $\mathrm{id}_V \in \mathrm{End}(V)$, so $\chi'$ is a subrepresentation of $V \otimes V$. Since $\chi'$ is $1$-dimensional, it is a weak abelian direct summand of $E$-rational semisimple representation $\rho_{\pi,\lambda_0} \otimes \rho_{\pi,\lambda_0}$. By Theorem~\ref{thm_weak abelian direct summand-locally algebraic}, $\chi'$ is locally algebraic, and hence so is $\chi$.

    To finish, we show that $(\xi_{\chi})_v(-1)$ is independent of $v \mid \infty$.
		This is obviously true if $K=\Q$ (condition \ref{thm_main thm wo polarizability}(a)), so we may assume that $n$ is odd (condition \ref{thm_main thm wo polarizability}(b)). 
		Since $\chi(c_v)=(\xi_{\chi})_v(-1)$ (cf. \cite[p. 511--512]{BLGGT14})
		and $\chi'=\epsilon^{1-n}\chi$, it is enough to prove $\chi'(c_v)$ is independent of $v \mid \infty$. By \eqref{eqn_self duality of Gal repn}, we have
    \begin{align*}
        \det(\rho_{\pi,\lambda_0}(c_v))=\det({}^t\rho_{\pi,\lambda_0}(c_v)^{-1})\chi'(c_v)^n.
    \end{align*}
    Thus, $\chi'(c_v)^n=\det(\rho_{\pi,\lambda_0}(c_v))^2=1$, and consequently, $\chi'(c_v)=1$ since $n$ is odd.
\end{proof}

\section{Proof of Theorem~\ref{thm_main theorem}}

Henceforth, we assume $K$ is a totally real field and $\pi$ is a regular algebraic polarized cuspdial automorphic 
representation of $\GL_n(\A_K)$. Let $\{\rho_{\pi, \lambda}:\mathrm{Gal}_{K} \to \mathrm{GL}_n(\overline E_{\lambda})\}_{\lambda}$ be the regular semisimple strictly compatible system of $K$ defined over $E$ attached to $\pi$.
Suppose the conditions Theorem~\ref{thm_main theorem}(a)--(c) are satisfied. 
 By Theorem~\ref{thm_Serre}(i), the algebraic monodromy group $\bG_{\lambda}$ is connected for all $\lambda$. 
 The proof of Theorem~\ref{thm_main theorem} proceeds in two steps.
The first step is to show that Theorem~\ref{thm_main theorem} follows 
from Proposition \ref{prop the same Lie type}, which asserts the $\lambda$-independence of 
the Lie type of $\bG_\lambda$ (resp. the root system of $\bG_\lambda^{\der}$);
the second step is to prove Proposition \ref{prop the same Lie type} by constructing sufficiently many three-dimensional automorphic compatible systems
(via the potential automorphy theorem) and using the fact that 
the semisimple rank ($r\in\Z_{\geq0}$) of $\bG_{\lambda}$ is independent of $\lambda$ (Theorem~\ref{thm_formal bi-char indep}).
From now on, we assume this semisimple rank $r$ is positive, otherwise, $n=1$, in which case Theorem~\ref{thm_main theorem}
is obvious.

\subsection{Reduction to Proposition~\ref{prop the same Lie type}} 
We embed $\overline E_\lambda$ into $\C$ for all $\lambda$.
Theorem~\ref{thm_main theorem}(i) states that the conjugacy class of the base change 
$\bG_{\lambda,\C}$ in $\GL_{n,\C}$
is independent of $\lambda$, and Theorem~\ref{thm_main theorem}(ii) follows immediately from Theorem~\ref{thm_main theorem}(i). Moreover, Theorem~\ref{thm_main theorem}(iii) follows immediately from Theorem~\ref{thm_main theorem}(ii), due to the following two results. We note that by enlarging $E$ and replacing our compatible system to its corresponding coefficient extension (Definition~\ref{def_coe ext and res scalar}(i)), 
we may assume that $\rho_{\pi,\lambda}$ is $\mathrm{GL}_n(E_{\lambda})$-valued as we mentioned in Section~\ref{sec_automorphic Gal repn}.

\begin{proposition}{\cite[Proposition 2.11]{Hui23b}}\label{prop_strict-BT and PS}
  Let $\{\rho_{\lambda}: \mathrm{Gal}_K \to \mathrm{GL}_n(E_{\lambda})\}_{\lambda}$ be a  strictly compatible system of $K$ defined over $E$. Then there exist some integers $N_1, N_2 \geq 0$ and a finite extension $K'/K$ such that the following two conditions hold.
  \begin{enumerate}[\normalfont(a)]
      \item (Bounded tame inertia weights) for almost all $\lambda$ and each finite place $v$ of $K$ above $\ell$, the tame inertia weights of the local representation $(\overline{\rho}_{\lambda}^{\mathrm{ss}} \otimes \overline{\epsilon}_{\ell}^{N_1})|_{\mathrm{Gal}_K}$ belongs to $[0,N_2]$.
      \item (Potential semistability) for almost all $\lambda$ and each finite place $w$ of $K'$ not above $\ell$, the semisimplification of the local representation $\overline{\rho}_{\lambda}^{\mathrm{ss}}|_{\mathrm{Gal}_K}$ is unramified.
  \end{enumerate}
\end{proposition}

\begin{theorem}{\cite[Theorem 1.2]{Hui23a}}
    Let $\{\rho_{\lambda}: \mathrm{Gal}_K \to \mathrm{GL}_n(E_{\lambda})\}_{\lambda}$ be a semisimple compatible system of $K$ defined over $E$ that satisfies the conditions (a) and (b) of Proposition~\ref{prop_strict-BT and PS}. For all but finitely many $\lambda$, if $\sigma_{\lambda}$ is a type $A$ irreducible subrepresentation of $\rho_{\lambda} \otimes_{E_{\lambda}} \overline{E}_{\lambda}$, then $\sigma_{\lambda}$ is residually irreducible.
\end{theorem}

To prove Theorem~\ref{thm_main theorem}(i),
we first state the following theorem on the representation theory
of a connected semisimple group due to Larsen--Pink, which we will apply to the irreducible representation
of $\bG_{\lambda_0}^{\der}$.

\begin{theorem}{\cite[Theorem 4]{LP90}}\label{thm_LP90}
    Let $T$ be a torus, and $\rho_T$ a faithful representation of $T$. Suppose that there exists a connected semisimple group $G$, a faithful irreducible representation $\rho$ of $G$, and an isomorphism between $T$ and a maximal torus of $G$, such that $\rho_T$ is the pull-back of $\rho$. Then the pair $(G,\rho)$ is uniquely determined up to isomorphism by $(T,\rho_T)$, except for the equivalences generated by the following exceptions.
    \begin{enumerate}[\normalfont(a)]
        \item The spin representation of $B_n$ restricts to the tensor product of the spin representations of $\sum B_{n_i} \subset B_n$, $\sum n_i=n$.
        \item For every $n \geq 2$ the representations of $C_n$ and $D_n$ with the highest weight $(k,k-1,\ldots,1,0,\ldots,0)$ for $1 <k\leq n-1$ have the same formal character.
        \item There are unique irreducible representations of $A_2$ and $G_2$ of dimension $27$, which have the same formal character.
        \item There are unique irreducible representations of $C_4$, $D_4$, and $F_4$, of dimension $4096=2^{12}$, which have the same formal character.
    \end{enumerate}
\end{theorem}

Although Theorem~\ref{thm_main theorem}(c2) belongs to the exceptional case (a) above,
Theorem~\ref{thm_main theorem}(i) in this case has been proven in \cite[Proposition 4.16]{Hui23a}.
Consequently, we need only consider Theorem~\ref{thm_main theorem}(c1), which avoids the exceptional cases in Theorem~\ref{thm_LP90}. 
Since Patrikis--Taylor (Theorem~\ref{thm_PT15}) asserts that $\rho_{\pi,\lambda}$ 
(and thus $\bG_\lambda$) is irreducible for 
 infinitely many $\lambda$, we obtain Proposition \ref{infinitelambda} below by
 Serre and Hui (Theorem~\ref{thm_Serre} and Theorem~\ref{thm_formal bi-char indep}) and Larsen--Pink (Theorem~\ref{thm_LP90}).

\begin{proposition}\label{infinitelambda}
Let $\{\rho_{\pi, \lambda}: \mathrm{Gal}_{K} \to \mathrm{GL}_n(\overline E_{\lambda})\}_{\lambda}$ be 
		the compatible system satisfying conditions (a)--(c) in Theorem~\ref{thm_main theorem}. There exist infinitely many
		$\lambda\in\Sigma_E$ such that 
		the reductive subgroups $\bG_{\lambda,\C}$ and $\bG_{\lambda_0,\C}$ 
are conjugate in $\GL_{n,\C}$.
\end{proposition}


 For a connected reductive group $G$ with a maximal torus $T$, denote 
by $R(G,T)$ the set of roots of $G$ with respect to $T$ in the character group $\mathbb{X}(T)$. 
We also need the following result for reductive groups.

\begin{proposition}\label{invroot}\cite[Corollary 3.9]{Hui18}
Let $G_1$ and $G_2$ be two connected reductive subgroup of $\GL_{n,\C}$ 
with maximal tori $T_1$ and $T_2$ respectively.
Denote by $T_i':=T_i\cap G_i^{\der}$ the maximal torus of the derived group $G_i^{\der}$ for $i=1,2$. 
If the chains $T_1'\subset T_1$ and $T_2'\subset T_2$ are conjugate in $\GL_{n,\C}$ (i.e., $G_1$ and $G_2$
have the same formal bi-character) such that the roots $R(G_1^{\der},T_1')\subset\mathbb{X}(T_1')$
and $R(G_2^{\der},T_2')\subset\mathbb{X}(T_2')$ are identified, then $G_1$ and $G_2$
are conjugate in $\GL_{n,\C}$.
\end{proposition}




We will first prove that $\bG_{\lambda}^{\mathrm{der}}$ is of type $A_1$ for all $\lambda$, 
regardless of whether $\rho_{\pi,\lambda}$ is irreducible. After that, we will prove
 that the conjugacy class of $\bG_{\lambda,\C}\subset\GL_{n,\C}$
is independent of $\lambda$ by using Proposition \ref{invroot}. 
In short, it suffices to prove the following proposition.
This will be done in the next section.
\begin{proposition}\label{prop the same Lie type}
    Let $\{\rho_{\pi, \lambda}: \mathrm{Gal}_{K} \to \mathrm{GL}_n(\overline E_{\lambda})\}_{\lambda}$ be 
		the compatible system satisfying conditions (a)--(c) in Theorem~\ref{thm_main theorem} and $\bT_\lambda\subset\bG_\lambda$
		a maximal torus for all $\lambda$. 
		The following assertions hold.
		\begin{enumerate}[(i)]
		\item The algebraic monodromy group $\bG_\lambda$ is of type $A_1$ for all $\lambda$. 
		\item For every $\lambda$, the formal bi-characters $\bT_{\lambda,\C}'\subset\bT_{\lambda,\C}$
		and $\bT_{\lambda_0,\C}'\subset\bT_{\lambda_0,\C}$ are conjugate in $\GL_{n,\C}$ such that 
		the roots $R(\bG_{\lambda,\C},\bT_{\lambda,\C})$ 
		and $R(\bG_{\lambda_0,\C},\bT_{\lambda_0,\C})$ are identified
		in the character groups $\mathbb{X}(\bT_{\lambda,\C})\simeq \mathbb{X}(\bT_{\lambda_0,\C})$.
		\end{enumerate}
\end{proposition}

\begin{remark}
Since the restriction map $\mathbb{X}(\bT_{\lambda,\C})\to \mathbb{X}(\bT_{\lambda,\C}')$ maps $R(\bG_{\lambda,\C},\bT_{\lambda,\C})$ bijectively to $R(\bG_{\lambda,\C}^{\der},\bT_{\lambda,\C}')$, the roots $R(\bG_{\lambda,\C}^{\der},\bT_{\lambda,\C}')$	and $R(\bG_{\lambda_0,\C}^{\der},\bT_{\lambda_0,\C}')$ are also identified in $\mathbb{X}(\bT_{\lambda,\C}')\simeq \mathbb{X}(\bT_{\lambda_0,\C}')$	by Proposition \ref{prop the same Lie type}(ii).
\end{remark}

\subsection{Proof of Proposition \ref{prop the same Lie type}}\label{subsec_Proof of Prop the same Lie type}

Suppose $H$ is a linear algebraic group over $\overline{\mathbb Q}_{\ell}$. A $H$-valued $\ell$-adic Galois representation $\rho: \mathrm{Gal}_{F_v} \to H(\overline{\mathbb Q}_{\ell})$
of a local field $F_v/\Q_\ell$ is said to be Hodge-Tate if for any morphism (of algebraic groups) $H \to \mathrm{GL}_{n,\overline{\mathbb Q}_{\ell}}$, the composite $\ell$-adic representation 
$$\rho: \mathrm{Gal}_{F_v} \to H(\overline{\mathbb Q}_{\ell}) \to \mathrm{GL}_n(\overline{\mathbb Q}_{\ell})$$ 
is a Hodge-Tate representation.

We first record the following proposition on the Hodge-Tate lift which will be used in the proof of Lemma \ref{lem_phi is regular}.

\begin{proposition}{\cite[Corollary 3.2.12]{Pat19}}\label{prop_Hodge-Tate lift}
    Let $H' \twoheadrightarrow H$ be a surjection of linear algebraic groups $H'$ and $H$ over $\overline{\mathbb Q}_{\ell}$ with the central torus kernel and $\rho: \mathrm{Gal}_{F_v} \to H(\overline{\mathbb Q}_{\ell})$ be a Hodge-Tate representation for a finite field extension $F_v/\mathbb Q_{\ell}$. Then there exists a Hodge-Tate lift $\tilde{\rho}:\mathrm{Gal}_{F_v} \to H'(\overline{\mathbb Q}_{\ell})$ so that the following diagram commutes.
    \[\begin{tikzcd}
	& {H'(\overline{\mathbb Q}_{\ell})} \\
	{\mathrm{Gal}_{F_v}} & {H(\overline{\mathbb Q}_{\ell})}
	\arrow[two heads, from=1-2, to=2-2]
	\arrow["{\tilde{\rho}}", dashed, from=2-1, to=1-2]
	\arrow["\rho"', from=2-1, to=2-2]
\end{tikzcd}\]
\end{proposition}

\vspace{.1in}

Let $\phi_{\lambda_0}$ be the Galois representation of $K$ given by 
the conjugate action of the image of $\rho_{\lambda_0}$ on $\mathrm{Lie}(\bG_{\lambda_0}^{\mathrm{der}})$ (adjoint representation). Since $\bG_{\lambda_0}^{\mathrm{der}}$ has rank $r\geq 1$ and of type $A_1$,
we obtain a decomposition
$$\phi_{\lambda_0}=\phi_{\lambda_0,1} \oplus \cdots \oplus \phi_{\lambda_0,r},$$ where each $\phi_{\lambda_0,j}$ is the $3$-dimensional adjoint representation on $\mathfrak{sl}_2$. 
The algebraic monodromy group of $\phi_{\lambda_0}$ is $\SO_3^r\subset\GL_{3r}$.

  
\begin{lemma}\label{lem_phi is regular}
    For $1\leq i \leq r$, the representation $\phi_{\lambda_0,i}$ is regular.
\end{lemma}

\begin{proof}
Recall that the adjoint representation on $\mathrm{Lie}(\bG_{\lambda_0}^{\mathrm{der}}) \leq \mathrm{Lie}(\mathrm{GL}_n(\overline{E}_{\lambda_0}))=\mathrm{End}(\overline{E}_{\lambda_0}^n)$ is a subrepresentation of the representation given by the conjugate action on $\mathrm{End}(\overline{E}_{\lambda_0}^n)$, namely, $\mathrm{End}(\rho_{\pi,\lambda_0})$. Since $\mathrm{End}(\rho_{\pi,\lambda_0})$ is isomorphic to $\rho_{\pi, \lambda_0} \otimes \rho_{\pi, \lambda_0}^{\vee}$, it follows that $\phi_{\lambda_0}$ is a subrepresentation of $\rho_{\pi, \lambda_0} \otimes \rho_{\pi, \lambda_0}^{\vee}$, which is de Rham at $v \mid \ell_0$. Thus, to verify regularity (Definition \ref{regularalg}), we need to show that the Hodge-Tate numbers of 
 $\phi_{\lambda_0,i}$ at $v \mid \ell_0$ are distinct. For each $v\mid \ell_0$, 
our strategy is to find two-dimensional Hodge-Tate Galois representations $f_i$ such that $\phi_{\lambda_0,i}|_{\mathrm{Gal}_{K_v}}=\mathrm{ad}^0(f_i)$ (trace zero adjoint representation) for all $1\leq i\leq r$. If we can show that $f_i$
has distinct Hodge-Tate numbers, then so does $\phi_{\lambda_0,i}|_{\mathrm{Gal}_{K_v}}$ which completes the proof.

By twisting the appropriate power of the system of $\ell$-adic cyclotomic characters to $\{\rho_{\pi,\lambda}\}_{\lambda}$, we may assume that 
$\bG_{\lambda_0}=\mathbb{G}_m\bG_{\lambda_0}^{\der}$.
Suppose $\bG_{\lambda_0}^{\der}\subset\GL_{n,\overline E_{\lambda_0}}$ is the image of 
 $$\mathrm{Sym}^{k_1} \otimes \cdots \otimes \mathrm{Sym}^{k_s} \otimes \mathrm{Sym}^{k_{s+1}} \otimes \cdots \otimes \mathrm{Sym}^{k_{s+t}}: \SL_2^r \to \GL_n$$ 
with $s+t=r$, $k_1,\ldots,k_s\geq 1$ odd and $k_{s+1},\ldots,k_{s+t}\geq 1$ even. 
When $k_i$ is even, the image $\mathrm{Sym}^{k_i}(\SL_2)$ is isomorphic to $\mathrm{SO}_3$ which is adjoint. 
Hence, we obtain a factorization  
\begin{align*}
   \bG_{\lambda_0}^{\der}= G_1\times G_2,
\end{align*}
where $G_1$ (if non-trivial) is a connected rank $s$ semisimple group of type $A_1$ whose center is of order $2$
and $G_2$ is isomorphic to $\SO_3^t$.
It follows that 
$$\bG_{\lambda_0}=\mathbb{G}_m\bG_{\lambda_0}^{\der}= (\mathbb{G}_m G_1) \times G_2.$$

Consider the following two central torus surjections
\begin{align*}
   \pi_1: \mathrm{GL}_2^s &\to  \mathbb{G}_m G_1, \\
    (g_1, \ldots, g_s) &\mapsto \bigotimes_{i=1}^s \det(g_{i})^{\frac{1-k_{i}}{2}}\mathrm{Sym}^{k_i}(g_i),
\end{align*}
and  
\begin{align*}
  \pi_2:  \mathrm{GL}_2^t &\to G_2\cong\SO_3^t,\\
    (g_{s+1}, \ldots, g_{s+t}) &\mapsto \bigotimes_{j=1}^t \det(g_{s+j})^{\frac{-k_{s+j}}{2}}\mathrm{Sym}^{k_{s+j}}(g_{s+j}),
\end{align*} 
which is well-defined since $k_1,\ldots,k_s$ are odd and $k_{s+1},\ldots,k_{s+t}$ are even.
Note that the kernel of $\pi_1$ and $\pi_2$ are respectively
$\{(c_1,...,c_s)\in\mathbb{G}_m^s:~ c_1\cdots c_s=1\}$ and $\mathbb{G}_m^t$.
These two surjections yield a central torus quotient 
\begin{equation}\label{pi12}
\pi_1\otimes\pi_2:\mathrm{GL}_2^r \to \bG_{\lambda_0}\subset\GL_{n,\overline E_{\lambda_0}}.
\end{equation}

By Proposition \ref{prop_Hodge-Tate lift}, for 
a prime $v$ above $\ell_0$, there is a Hodge-Tate lift $f_1 \oplus \cdots \oplus f_r: \mathrm{Gal}_{K_v} \to \mathrm{GL}_2(\overline{E}_{\lambda_0})^r$ and consequently, we have
\begin{align}\label{eqn_rho0 tensor product decomposition}
    \rho_{\pi, \lambda_0}|_{\mathrm{Gal}_{K_v}}=\bigotimes_{i=1}^s \det(f_{i})^{\frac{1-k_{i}}{2}}\mathrm{Sym}^{k_i}(f_i) \otimes \bigotimes_{j=1}^t \det(f_{s+j})^{\frac{-k_{s+j}}{2}}\mathrm{Sym}^{k_{s+j}}(f_{s+j})
\end{align}

Since \eqref{eqn_rho0 tensor product decomposition} has distinct Hodge-Tate numbers, the two dimensional $f_i$ and the three dimensional $\phi_{\lambda_0,i}|_{\mathrm{Gal}_{K_v}}=\ad^0(f_i)$ also have distinct Hodge-Tate numbers 
because the isomorphism $\mathrm{Lie}(\SL_2^r)\simeq \mathrm{Lie}(\bG_{\lambda_0}^{\der})$
induced by \eqref{pi12} is $\Gal_{K_v}$-equivariant.
\end{proof}

\begin{proof}[Proof of Proposition \ref{prop the same Lie type}(i)]
By Proposition \ref{infinitelambda}, we may assume $\ell_0$ is sufficiently large so that 
 Lemma \ref{lem_phi is regular} and Proposition \ref{prop_SO3 potential automorphy}
imply that $\phi_{\lambda_0,i}$ belongs to a strictly compatible system $\{\phi_{\lambda,i}\}_{\lambda}$ for all $1 \leq i \leq r$, after coefficient extension (Definition \ref{def_coe ext and res  scalar}) if necessary.
Thus we obtain a strictly compatible system 
\begin{equation}\label{csphi}
\{\phi_{\lambda}:=\phi_{\lambda,1}\oplus \cdots \oplus \phi_{\lambda,r}\}_{\lambda}
\end{equation}
with algebraic monodromy groups $\{\bG_{\phi_{\lambda}}\}_\lambda$. 
Since $\bG_{\phi_{\lambda}}\subset\SO_3^r\subset\GL_3^r$ for all $\lambda$
and $\bG_{\phi_{\lambda_0}}=\mathrm{SO}_3^r$, 
it follows from Theorem \ref{thm_formal bi-char indep} that
\begin{equation}\label{equal}
\bG_{\phi_{\lambda}}=\mathrm{SO}_3^r\hspace{.2in}\text{for all}~\lambda.
\end{equation}


Consider the strictly compatible system of $\ell$-adic representations
\begin{align*}
    \{\psi_{\lambda}:=\rho_{\pi, \lambda}\oplus \phi_{\lambda}\}_{\lambda}
\end{align*}
of $K$ defined over $E$ with the algebraic monodromy groups $\{\bH_{\lambda}\}_\lambda$. 
Since $\phi_{\lambda_0}$ is a subrepresentation of $\rho_{\pi, \lambda_0} \otimes \rho_{\pi, \lambda_0}^{\vee}$, it follows that $\bH_{\lambda_0}\simeq \bG_{\lambda_0}$ is connected with semisimple rank $r$. 
Hence for all $\lambda$, 
$\bH_{\lambda}$ is connected with semisimple rank $r$ (and reductive rank $r+1$) 
by Theorems \ref{thm_Serre} and \ref{thm_formal bi-char indep}.

Since $(\alpha_\lambda,\beta_\lambda):\bH_{\lambda}\hto \bG_\lambda\times\bG_{\phi_{\lambda}}$ that projects onto 
both $\bG_\lambda$ and $\bG_{\phi_{\lambda}}$, we obtain by Goursat's lemma an isomorphism
\begin{align}\label{eqn_Goursat isomorphism}
   \bG_{\lambda}/\iota_1^{-1}(\bH_{\lambda})\stackrel{\simeq}{\longrightarrow}  \bG_{\phi_{\lambda}}/\iota_2^{-1}(\bH_{\lambda})
\end{align}
where $\iota_1:\bG_\lambda \hookrightarrow \bG_\lambda\times\bG_{\phi_{\lambda}}$ 
and $\iota_2: \bG_{\phi_{\lambda}} \hookrightarrow \bG_\lambda\times\bG_{\phi_{\lambda}}$ 
are canonical embeddings, and also short exact sequences 
\begin{equation}\label{ses1}
1\to \iota_2^{-1}(\bH_{\lambda})\to \bH_{\lambda}\stackrel{\alpha_\lambda}{\rightarrow} \bG_{\lambda}\to 1
\end{equation}
and
\begin{equation}\label{ses2}
1\to \iota_1^{-1}(\bH_{\lambda})\to \bH_{\lambda}\stackrel{\beta_\lambda}{\rightarrow} \bG_{\phi_{\lambda}}\to 1.
\end{equation}
Since $\bH_{\lambda}$ and $\bG_{\phi_{\lambda}}$ have semisimple rank $r$, 
\eqref{ses2} implies that the identity component of $\iota_1^{-1}(\bH_{\lambda})$ is a torus. Then \eqref{eqn_Goursat isomorphism}
and \eqref{equal} imply that $\bG_{\lambda}$ is of type $A_1$.
\end{proof}




\begin{proof}[Proof of Proposition \ref{prop the same Lie type}(ii)]
Since $\bH_\lambda$ and $\bG_\lambda$ have the same rank and semisimple rank,
\eqref{ses1} implies that $\iota_2^{-1}(\bH_{\lambda})$ is finite.
As a finite normal subgroup of  $\bG_{\phi_{\lambda}}\simeq \SO_3^r$ (adjoint group) by \eqref{eqn_Goursat isomorphism},
$\iota_2^{-1}(\bH_{\lambda})$ must be trivial and $\alpha_\lambda:\bH_\lambda\to\bG_\lambda$ is an isomorphism.


Let $\bT_\lambda'\subset\bT_\lambda\subset\GL_n\times\GL_{3r}$ such that $\bT_\lambda$ (resp. $\bT_\lambda'$)
is a maximal torus of $\bH_\lambda$ (resp. $\bH_\lambda^{\der}$) for all $\lambda$.
Since $\{\psi_\lambda=\rho_{\pi,\lambda}\oplus\phi_\lambda\}_\lambda$ is a sum of two compatible systems,
the two chains
$\bT_{\lambda,\C}'\subset\bT_{\lambda,\C}$ and $\bT_{\lambda_0,\C}'\subset\bT_{\lambda_0,\C}$
are conjugate in $\GL_n(\C)\times\GL_{3r}(\C)$.
For all $\lambda$, consider the morphism
\begin{align*}
    \theta_\lambda: \bG_{\lambda} \stackrel{\alpha_\lambda^{-1}}{\rightarrow} 
		\bH_{\lambda} \stackrel{\alpha_\lambda\oplus\beta_\lambda}{\hto} 
		\bG_\lambda\times\mathrm{SO}_3^r \subset \GL_n\times \mathrm{GL}_{3r}.
\end{align*}
By Proposition \ref{prop the same Lie type}(i), the representation 
$\beta_\lambda\circ \alpha_\lambda^{-1}:\bG_{\lambda}\to\GL_{3r}$ (surjecting onto $\SO_3^r$) can only be 
the adjoint representation
on $\mathrm{Lie}(\bG_\lambda^{\der})$.
Therefore, $\alpha_{\lambda}(\bT_{\lambda,\C}')\subset\alpha_{\lambda}(\bT_{\lambda,\C})$
is conjugate to $\alpha_{\lambda_0}(\bT_{\lambda_0,\C}')\subset\alpha_{\lambda_0}(\bT_{\lambda_0,\C})$
such that roots $R(\bG_{\lambda,\C},\alpha_{\lambda}(\bT_{\lambda,\C}))$ 
and $R(\bG_{\lambda_0,\C},\alpha_{\lambda_0}(\bT_{\lambda_0,\C}))$ are identified,
as weights of $\theta_\lambda$ and $\theta_{\lambda_0}$ in the character groups $$\mathbb{X}(\alpha_{\lambda}(\bT_{\lambda,\C}))\simeq \mathbb{X}(\alpha_{\lambda_0}(\bT_{\lambda_0,\C}))$$ coincide.
\end{proof}

\section{Proof of Theorem~\ref{thm_main theorem-classification}}

In the previous section, we referred to Proposition~\ref{prop_Hodge-Tate lift} on the Hodge-Tate lift of local representation due to Patrikis. The following is an analogous theorem of this proposition, the global version of the crystalline lift, deduced from Patrikis' results \cite{Pat15, Pat16, Pat19}.

\begin{theorem}\label{thm_crystalline lift}
    Let $H' \twoheadrightarrow H$ be a surjection of linear algebraic groups $H'$ and $H$ over $\overline{\mathbb Q}_{\ell}$ with the central torus kernel and $\rho: \mathrm{Gal}_{\mathbb Q} \to H(\overline{\mathbb Q}_{\ell})$ be a continuous Galois representation that is unramified outside a finite set of primes and such that $\rho |_{\mathrm{Gal}_{\mathbb Q_{\ell}}}$ crystalline. Then there exists a Galois representation $\tilde{\rho}:\mathrm{Gal}_{\mathbb Q} \to H'(\overline{\mathbb Q}_{\ell})$ such that $\tilde{\rho}|_{\mathrm{Gal}_{\mathbb Q_{\ell}}}$ is crystalline and the following diagram commutes.
\[\begin{tikzcd}
	& {H'(\overline{\mathbb Q}_{\ell})} \\
	{\mathrm{Gal}_{\mathbb Q}} & {H(\overline{\mathbb Q}_{\ell})}
	\arrow[two heads, from=1-2, to=2-2]
	\arrow["{\tilde{\rho}}", dashed, from=2-1, to=1-2]
	\arrow["\rho"', from=2-1, to=2-2]
\end{tikzcd}\]
\end{theorem}

\begin{proof}
    The proof is given in \cite[Theorem 2.13]{DWW24}.
\end{proof}

Suppose that $\{\rho_{\pi, \lambda}: \mathrm{Gal}_{\mathbb Q} \to \mathrm{GL}_n(\overline E_{\lambda})\}_{\lambda}$ is the compatible system attached to automorphic representation $\pi$ as stated in Theorem~\ref{thm_main theorem-classification}. By Proposition \ref{prop_essential self-duality}, $\pi$ is polarized
and $\{\rho_{\pi, \lambda}\}_\lambda$ is actually strictly compatible.
We may assume that $\ell_0$ is sufficiently large so that $\rho_{\pi,\lambda_0}|_{\Gal_{K_v}}$
is crystalline for all $v$ above $\ell_0$ since, by Theorem~\ref{thm_main thm wo polarizability}, $\bG_{\lambda,\C}\subset\GL_{n,\C}$ is 
independent of $\lambda$.
By twisting $\{\rho_{\pi, \lambda}\}_\lambda$ with $\{\epsilon_\ell^k\}_\ell$, which is the compatible system of some power of $\ell$-adic cyclotomic characters, we assume the algebraic monodromy $\bG_{\lambda}$ is equal to $\mathbb{G}_m \bG_\lambda^{\der}$
for all $\lambda$.

By \eqref{pi12} and  Theorem~\ref{thm_crystalline lift}, we obtain 
$f_{1,\lambda_0}\oplus\cdots\oplus f_{r,\lambda_0}:\Gal_\Q\to \GL_2^r(\overline E_{\lambda_0})$
and a character $\chi_{\lambda_0}:\Gal_\Q\to\GL_1(\overline E_\lambda)$
that are crystalline at $\ell_0$ such that 
\begin{align}\label{eqn_crystalline tensor decomposition}
 \rho_{\pi, \lambda_0}= \mathrm{Sym}^{k_1}(f_{1, \lambda_0}) \otimes \cdots \otimes \mathrm{Sym}^{k_r}(f_{r, \lambda_0}) \otimes \chi_{\lambda_0},
\end{align}
where $k_1,...,k_r\geq 1$ are independent of $\lambda_0$.
Moreover, $f_{i,\lambda_0}$ is regular for all $i$ since $\rho_{\pi, \lambda_0}$ is regular
and the algebraic monodromy group of $f_{1,\lambda_0}\oplus\cdots\oplus f_{r,\lambda_0}$ is of type $A_1$.


We claim that these representations $f_{1, \lambda_0},\ldots,f_{r, \lambda_0}$ are modular up to twists for sufficiently large $\ell_0$. To support our argument, we record the following two results; a criterion concerning the oddness of a continuous Galois representation, and the Fontaine-Mazur conjecture in the regular case which has been proved by Skinner--Wiles \cite{SW01}, Kisin \cite{Kis09}, Hu--Tan \cite{HT15}, and most recently by Pan \cite{Pan22}.

\begin{proposition}{\cite[Proposition 2.5]{CG13}}\label{prop_CG13}
Let $F$ be a totally real field. Suppose that $\ell>7$ is prime, and that $\rho: \mathrm{Gal}_F \to \mathrm{GL}_2(\overline{\mathbb Q}_{\ell})$ is a continuous representation. Assume the following conditions.
\begin{enumerate}[\normalfont(a)]
    \item $\rho$ is unramified outside of finitely many primes.
    \item $\mathrm{Sym}^2 \overline{\rho}^{\mathrm{ss}}|_{\mathrm{Gal}_{F(\zeta_{\ell})}}$ is irreducible.
    \item $\ell$ is unramified in $F$.
    \item For each place $v \mid \ell$ of $F$ and each $\tau: F_v \hookrightarrow \overline{\mathbb Q}_{\ell}$, $\mathrm{HT}_{\tau}(\rho|_{\mathrm{Gal}_{F_v}})$ is a set of $2$ distinct integers whose difference is less than $(\ell-2)/2$, and $\rho|_{\mathrm{Gal}_{F_v}}$ is crystalline.
\end{enumerate}
Then the pair $(\rho,\det \rho)$ is polarized and odd.
\end{proposition}

\begin{theorem}{\cite[Theorem 1.0.4]{Pan22}}\label{thm_Fontaine-Mazur}
    Let $\ell$ be an odd prime and $\rho: \mathrm{Gal}_{\mathbb Q} \to \mathrm{GL}_2(\overline{\mathbb Q}_{\ell})$ be a continuous irreducible representation such that it is odd, regular and potentially semi-stable at $\ell$. Then $\rho$ is modular up to twist.  
\end{theorem}

\begin{proof}[Proof of Theorem~\ref{thm_main theorem-classification}]
It suffices to show that $f_{i,\lambda_0}$ is modular up to twist for $1\leq i \leq r$.
     If we can show that $f_{i,\lambda_0}$ is odd for $1\leq i \leq r$, we are 
		done by Theorem~\ref{thm_Fontaine-Mazur}.
By Proposition \ref{prop_CG13}, it suffices to verify conditions (a)--(d) for $f_{i,\lambda_0}$.
Conditions (a) and (c) are obvious. Condition (d) follows from the strict compatibility of $\{\rho_{\pi,\lambda}\}_{\lambda}$ (specifically Definition \ref{def_weakly compatible system}(b)) and \eqref{eqn_crystalline tensor decomposition}. For (b), note that the algebraic monodromy of $f_{i,\lambda_0}$ contains $\SL_2$ 
and $\mathrm{Sym}^2 f_{i,\lambda_0}$ up to character twist 
is a subrepresentation of $\rho_{\pi,\lambda_0}\otimes \rho_{\pi,\lambda_0}^\vee$.
Since $\{\rho_{\pi,\lambda}\otimes \rho_{\pi,\lambda}^\vee\}_\lambda$ is strictly compatible, 
 Proposition \ref{prop_strict-BT and PS} and \cite[Theorem 3.12(v)]{Hui23a} 
imply that $\mathrm{Sym}^2 \overline{f}_{i,\lambda_0}^{\mathrm{ss}}|_{\mathrm{Gal}_{\Q(\zeta_{\ell_0})}}$ is irreducible when $\ell_0$ is sufficiently large.
\end{proof}


\begin{remark}
    Theorem~\ref{thm_main theorem-classification} can be proven directly using the Serre conjecture, which was completely proved in \cite{KW09a,KW09b}, rather than relying on Theorem~\ref{thm_Fontaine-Mazur}. For further elaboration on this approach, we refer the reader to \cite[$\mathsection$4.5.1]{Hui23a}.
\end{remark}

\section*{Acknowledgments}
We would like to express our gratitude to Stefan Patrikis for answering our questions regarding the lifting of Galois representations, for suggesting Remark \ref{SP}, and for his interest in this work. We are grateful to Ariel Weiss for his interest and comments.
We thank Michael Larsen and Lian Duan for their interest in this work. We also appreciate the referee for helpful comments and suggestions.
C.-Y. Hui was partially supported by Hong Kong RGC (no. 17314522) and NSFC (no. 12222120).


\begin{thebibliography}{}


\bibitem[BLGGT14]{BLGGT14}
T. Barnet-Lamb, T. Gee, D. Geraghty, R. Taylor, Richard: Potential automorphy and change of weight, Ann. of Math. {\bf 179} (2014), no. 2, 501--609.

\bibitem[BH24]{BH24}
G. B\"{o}ckle, C. Y. Hui: Weak abelian direct summands and irreducibility of Galois representations, arxiv:2404.08943.


\bibitem[CG13]{CG13}
F. Calegari, T. Gee: Irreducibility of automorphic Galois representations of $\mathrm{GL}(n)$, $n$ at most $5$, Ann. Inst. Fourier (Grenoble) {\bf 63} (2013), no. 5, 1881--1912.

\bibitem[Clo90]{Clo90}
L. Clozel: Motifs et formes automorphes: applications du principe de fonctorialit\'{e}, Automorphic forms, Shimura varieties, and $L$-functions, Vol. I (Ann Arbor, MI, 1988), Perspect. Math., vol. 10,
Academic Press, Boston, MA, 1990, pp. 77--159.


\bibitem[DWW24]{DWW24}
L. Duan, X. Wang, A. Weiss: Five-dimensional compatible systems and the Tate conjecture for elliptic surfaces, arxiv:2406.03617.



\bibitem[FM95]{FM95}
J. M. Fontaine, B. Mazur: Geometric Galois representations, Elliptic curves, modular forms, and Fermat's last theorem (J. Coates and S.-T. Yau, eds.), International Press (1995).


\bibitem[HLTT16]{HLTT16}
M. Harris, K. W. Lan, R. Taylor, J. Thorne:
On the rigid cohomology of certain Shimura varieties, \textit{Res. Math. Sci.} \textbf{3} (2016), article no. 37, 308 pp.

\bibitem[HT15]{HT15}
Y. Hu, F. Tan: The Breuil-M\'{e}zard conjecture for non-scalar split residual representations, Ann. Sci. Ec. Norm. Sup\'{e}r. (4) {\bf 48} (2015), no. 6, 1383--1421.

\bibitem[Hui13]{Hui13}
C. Y. Hui: Monodromy of Galois representations and equal-rank subalgebra equivalence, Math. Res. Lett. {\bf 20} (2013), no. 4, 705--728.

\bibitem[Hui18]{Hui18}
C. Y. Hui: On the rationality of certain type A Galois representations, Trans. Amer. Math. Soc. {\bf 370} (2018), no. 9, 6771--6794.

\bibitem[Hui23a]{Hui23a}
C. Y. Hui: Monodromy of subrepresentations and irreducibility of low degree automorphic Galois representations, J. Lond. Math. Soc. (2) {\bf 108} (2023), no. 6, 2436--2490.

\bibitem[Hui23b]{Hui23b}
C. Y. Hui: Monodromy of four-dimensional irreducible compatible systems of $\mathbb Q$, Bull. Lond. Math. Soc. {\bf 55} (2023), no. 4, 1773--1790.


\bibitem[JS81]{JS81}
H. Jacquet, J. A. Shalika: On Euler products and the classification of automorphic forms. II, Amer. J. Math. {\bf 103} (1981), no. 4, 777--815.

\bibitem[KW09a]{KW09a}
C. Khare, J. -P. Wintenberger: Serre's modularity conjecture. I, Invent. Math. {\bf 178} (2009), no. 3, 485--504.

\bibitem[KW09b]{KW09b}
C. Khare, J. -P. Wintenberger: Serre's modularity conjecture. II, Invent. Math. {\bf 178} (2009), no. 3, 505--586.



\bibitem[Kis09]{Kis09}
M. Kisin: The Fontaine-Mazur conjecture for $\mathrm{GL}_2$, J. Amer. Math. Soc. {\bf 22} (2009), no. 3, 641--690.

\bibitem[Lan79]{Lan79}
R. P. Langlands: Automorphic representations, Shimura varieties, and motives. Ein M\"{a}rchen. Automorphic forms, representations and $L$-functions II AMS (1979).

\bibitem[LP90]{LP90}
M. Larsen, R. Pink: Determining representations from invariant dimensions, Invent. Math. {\bf 102} (1990), no. 2, 377--398.

\bibitem[LP92]{LP92}
M. Larsen, R. Pink: On $\ell$-independence of algebraic monodromy groups in compatible systems of representations, Invent. Math. {\bf 107} (1992), no. 3, 603--636.



\bibitem[Pan22]{Pan22}
L. Pan: The Fontaine-Mazur conjecture in the residually reducible case, J. Amer. Math. Soc. {\bf 35} (2022), no. 4, 1031--1169.

\bibitem[Pat15]{Pat15}
S. Patrikis: On the sign of regular algebraic polarizable automorphic representations, Math. Ann. {\bf 362} (2015), no. 1-2, 147--171.

\bibitem[Pat16]{Pat16}
S. Patrikis: Generalized Kuga-Satake theory and rigid local systems, II: rigid Hecke eigensheaves, Algebra Number Theory {\bf 10} (2016), no. 7, 1477--1526.

\bibitem[Pat19]{Pat19}
S. Patrikis: Variations on a theorem of Tate, Mem. Amer. Math. Soc. {\bf 258} (2019), no. 1238, viii+156 pp.

\bibitem[PT15]{PT15}
S. Patrikis, R. Taylor: Automorphy and irreducibility of some $\ell$-adic representations, Compos. Math. {\bf 151} (2015), no. 2, 207--229.

\bibitem[Ram08]{Ram08}
  D. Ramakrishnan:
	Irreducibility and cuspidality. Representation theory and automorphic forms, 1--27, 
	Progr. Math., 255, Birkh$\mathrm{\ddot{a}}$user Boston, Boston, MA, (2008).
	
\bibitem[Sch15]{Sch15}
 P. Scholze:
	On torsion in the cohomology of locally symmetric varieties, \textit{Ann. of Math.} (2), 
	\textbf{182} (2015), no. 3, 945--1066.
	
\bibitem[Ser81]{Ser81}
J.-P. Serre: Lettre \'{a} Ken Ribet du 1/1/1981 et du 29/1/1981 (Oeuvres IV, no. 133)

\bibitem[Ser98]{Ser98}
J.-P. Serre: Abelian $\ell$-adic representations and elliptic curves, Res. Notes Math., 7
A K Peters, Ltd., Wellesley, MA, 1998.

\bibitem[SW01]{SW01}
C. M. Skinner, A. Wiles: Nearly ordinary deformations of irreducible residual representations, Ann. Fac. Sci. Toulouse Math. (6) {\bf 10} (2001), no. 1, 185--215.

\bibitem[Tay04]{Tay04}
R. Taylor: Galois representations, Ann. Fac. Sci. Toulouse Math. (6) {\bf 13} (2004), no. 1 , 73--119.

\bibitem[Var18]{Var18}
I.	Varma:
	Local-global compatibility for regular algebraic cuspidal automorphic representations when $\ell\neq p$, to appear in \textit{Forum of Mathematics Sigma}.

\end{thebibliography}
\end{document}